\documentclass[twoside,12pt]{article}
\usepackage{amsmath,amstext,amsgen,amsbsy,amsopn}
\usepackage{amssymb,amsfonts}
\usepackage{amsthm}
\usepackage{latexsym,amsxtra,euscript,amscd}

\usepackage{float}
\usepackage{xcolor}

\usepackage{booktabs}
\newcommand{\ra}[1]{\renewcommand{\arraystretch}{#1}}

\usepackage[colorlinks=true,urlcolor=blue,linkcolor=black]{hyperref}

\usepackage{graphics,graphicx}
\graphicspath{{Fig/},{Fig3D/},{../Nemeth_tiling_walking}}
\newtheorem{theorem}{Theorem}
\newtheorem{lemma}{Lemma}[section]
\newtheorem{corollary}{Corollary}

\newcommand*{\ADRnl}{ORCID: 0000-0001-9062-9280, University of Sopron,  Institute of Basic Sciences, Departement of Mathematics, Bajcsy-Zs.~u.~4.\ 9400, Sopron, Hungary. \texttt{nemeth.laszlo@uni-sopron.hu}}

\newcommand*{\TIT}{Walks on tiled boards}
\title{\bf \TIT}

\author{L\'aszl\'o N\'emeth\footnote{\ADRnl}}
\date{}

\begin{document}
	
	\maketitle \thispagestyle{empty}
	
	\begin{abstract}
		Several articles deal with tilings with various shapes, and also a very frequent type of combinatorics is to examine the walks on graphs or on grids. We combine these two things and give the numbers of the shortest walks crossing the tiled $(1\times n)$ and $(2\times n)$ square grids by covering them with squares and dominoes. We describe these numbers not only recursively, but also as rational polynomial linear combinations of Fibonacci numbers.
		\\[1mm]
		The  final  publication  is  available  via  journal   \href{https://www.degruyter.com/journal/key/ms/html#issues}{Mathematica Slovaca}. \\[1mm]
	\end{abstract}
	
	

	\section{Introduction}
	
	Hundreds, if not thousands, of articles can be found about tilings and walks on graphs, but to the best of our knowledge, there are not any paper that does both at the same time. In this article, we connect these two subfields of combinatorics as follows.

	Let $t_n$ be the number of the different tilings ($n$-tilings) with $(1\times 1)$-squares and $(1\times 2)$-dominoes  (two squares with a common edge) of a $(2\times n)$-board or square grid. Hereafter, a square always means $(1\times 1)$-square and a domino means $(1\times 2)$-domino.
	It is known, e.g., in  \cite{BQ}, that the number of possible tilings of a $(1\times n)$-board is $t_n=F_{n+1}$, where $(F_n)$ denotes the Fibonacci sequence (defined by $F_n=F_{n-1}+F_{n-2}$, $F_0=0$, $F_1=1$ and {A000045} in OEIS \cite{oeis}). 	 
	McQuistan and Lichtman \cite{ML} (generalizations in \cite{Kah}) studied the $(2\times n)$-board tilings with squares and dominoes in case of the Euclidean square mosaic, and they proved that $t_n$ satisfies the third-order recurrence relation $t_{n}=3t_{n-1}+t_{n-2}-t_{n-3}$ ({A030186} in OEIS \cite{oeis}). 
	Some researchers generalized the tiling with colored shapes \cite{Belb,Ben}.  Moreover, Komatsu, N\'emeth, and Szalay \cite{KNSz} examined the tilings with colored squares and dominoes on the hyperbolic $(2\times n)$-board, and they gave the fourth order linear homogeneous recurrence relation of $t_n$, where the coefficients satisfy recurrence sequences, as well.
	
	A self-avoiding walk on a graph is a walk that never visits the same vertex more than once. In the literature, self-avoiding walks have been studied mostly on infinite graphs. A central problem in this area is to consider the  number of distinct self-avoiding walks of a prescribed length on regular grids (see book of Madras and Slade \cite{mad} or  articles by Benjamin \cite{benja} and Williams \cite{will}). 
	However, self-avoiding walks have also been studied on finite graphs, such as square grids, i.e., \cite{bmg,bur,mad1,whi}, rectangular grids by Abbot \cite{Abbot} or complete graphs by Slade \cite{sla}. Major~et~al.\  \cite{MNSz} examined the self-avoiding walks on square grids in a special way when they defined forbidden direction steps, and they gave the numbers of walks in case of some special grids.
	
	In this article, we combine both of the above-mentioned combinatorial examinations. We consider square grids or boards and all the possible tilings with objects of them, and then we calculate the numbers of the shortest walks  crossing the tiled grids. 
	In Figure~\ref{fig:Til_walk_example_0s} there are four examples with tiled boards and walks crossing the boards. 	
	
	\begin{figure}[!ht]
		\centering
		\includegraphics[scale=0.9]{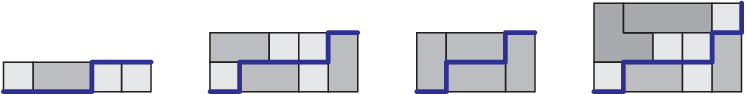} 
		\caption{Examples for walks on tiled boards}
		\label{fig:Til_walk_example_0s}
		\textbf{}\end{figure}
	
	In particular, we consider all the tilings with squares and dominoes on the boards $(1\times n)$ and $(2\times n)$ respectively, and then we give the number of the shortest walks crossing the tiled boards.   Obviously, these walks are self-avoiding walks and their lengths are $n+1$ and $n+2$ square-side-long, respectively. Of course, a walk must not cross any dominoes.
	
	We say that a walk crosses a board, if the starting and the ending points are at opposite corners of the board. 
	Any walk starts at the left-down vertex and ends at its right-up vertex. 
	Figure~\ref{fig:Til_walk_example_2x3} presents all the walks on the $2\times3$-board in case of a tiling.
	
	\begin{figure}[!ht]
		\centering
		\includegraphics[scale=0.9]{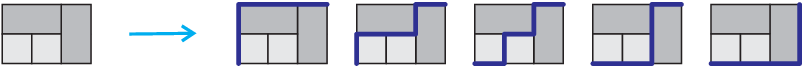} 
		\caption{Walks on $2\times3$-board with a given tiling}
		\label{fig:Til_walk_example_2x3}
	\end{figure}
	
	Let $(v_n)$ and $(w_n)$ be the so-called tiling-walking sequences, where $v_n$ and  $w_n$ respectively give the numbers of all the walks on all the tiled boards in case of $(1\times n)$ and $(2\times n)$, as well.
	
	The following theorems and corollary summarize our main results for the $(1\times n)$-board and $(2\times n)$-board. Since the numbers of the tilings with squares and dominoes on $(1\times n)$-board and with exclusively dominoes on $(2\times n)$-board are the Fibonacci numbers, then we describe the tiling-walking sequences in both cases  with the Fibonacci number as well in the corollaries. Furthermore, we mention that Liptai~et~al.\  \cite{Fibo_rep01} described some families of integer recurrence sequences as rational polynomial linear combinations of Fibonacci numbers. 
	\begin{theorem}\label{th:main_1xn}
		The tiling-walking sequence $(v_n)_{n=0}^\infty$ of the $(1\times n)$-board with squares and dominoes has the recurrence relation
		\begin{equation}\label{eq:main_1xn}
			nv_{n}=(n+1)v_{n-1}+(n+2)v_{n-2}, \quad n\geq 2,
		\end{equation} 
		where the initial values are $v_0=1$, $v_1=2$ (see {A001629} in OEIS\cite{oeis}). 
	\end{theorem} 
	
	\begin{corollary}\label{cor:main_1xn}
		The tiling-walking sequence $(v_n)_{n=0}^\infty$ of the $(1\times n)$-board with squares and dominoes is recursively given by the $4$th order constant coefficients homogeneous linear recurrence relation
		\begin{equation*}\label{cor:v}
			v_{n}=2v_{n-1}+v_{n-2}-2v_{n-3}-v_{n-4}, \quad n\geq 4,
		\end{equation*} 
		where the initial values are $v_0=1$, $v_1=2$, $v_2=5$, and $v_3=10$.
		Moreover, for $ n\geq 0$, we obtain the equation
		\begin{equation*}
			5\,v_{n}=2(n+2)\,F_{n+1}+(n+1)\,F_{n+2}, 
		\end{equation*}
		where $F_{n}$ is the $n$th Fibonacci number.
	\end{corollary}
	
	\begin{theorem}\label{th:main_2xn}
		The tiling-walking sequence $(w_n)_{n=0}^\infty$ of the $(2\times n)$-board with squares and dominoes is recursively given by the $9$th order homogeneous linear recurrence relation $(n\geq 9)$
		\begin{equation}\label{eq:main_2xn}
			w_{n}=8w_{n-1}-17w_{n-2}-7w_{n-3}+41w_{n-4}+w_{n-5}-23w_{n-6} +3w_{n-7}+4w_{n-8}-w_{n-9}
		\end{equation}
		with initial values $1, 5, 28, 130, 569, 2352, 9363, 36183, 136663$ ($n=0,\ldots,8$).\\
		A composed form  of \eqref{eq:main_2xn} is 
		\begin{equation*}
			r_{n}=3r_{n-1}+r_{n-2}-r_{n-3}
		\end{equation*}
		with 
		\begin{equation*}
			\begin{array}{rcl}
				r_{n}&=& x_n-3x_{n-1}-x_{n-2}+x_{n-3},\\
				x_{n}&=& y_n-3y_{n-1}+y_{n-2},\\
				y_{n}&=& w_n+w_{n-1}.
			\end{array}
		\end{equation*}
		
	\end{theorem}

	\begin{corollary}\label{th:maincor_2xn}
		The tiling-walking sequence $(w_n)_{n=0}^\infty$ of the $(2\times n)$-board with only dominoes is recursively given by the $6$-th order homogeneous linear recurrence relation $(n\geq 6)$
		\begin{equation}\label{eq:maincor_2xn}
			w_{n}=2w_{n-1}+2w_{n-2}-4w_{n-3}-2w_{n-4}+2w_{n-5}+w_{n-6}
		\end{equation}
		with initial values $1, 2, 6, 12, 26, 50$ for $n=0,\ldots,6$.  (See {A054454} in OEIS\cite{oeis}.)
		Moreover, for $ n\geq 0$, we obtain the equation
		\begin{eqnarray*}     
			w_{2n}= \frac{1+(-1)^n}{2}  + \frac{3}{5}(1+n)F_n + \frac{4n}{5} F_{n+1},
		\end{eqnarray*}
		and for even and odd terms, the equations   
		\begin{eqnarray*}
			5\,w_{2n} &=&  5+ (3+6n)\,F_{2n} + 8n\,F_{2n+1},\\
			5\,w_{2n+1} &=&  (6+6n)\, F_{2n+1} + (4+8n)\, F_{2n+2},
		\end{eqnarray*}
		where $F_{n}$ is the $n$th Fibonacci number. 
	\end{corollary}

	\section{Tiling and walking on \texorpdfstring{$(1\times n)$} --board}
	
	First, we show the walks  on the tiled  $(1\times 0)$-, $(1\times 1)$-, and $(1\times 2)$-boards in Figure~\ref{fig:Til_walk_1n_first}. Thus, $v_0=1$, $v_1=2$, and $v_2=5$.
	
	\begin{figure}[!ht]
		\centering
		\includegraphics[scale=0.9]{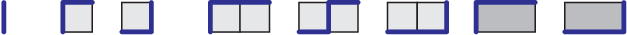} 
		\caption{All the walks on the tiled boards $1\times 0$, $1\times 1$, and $1\times 2$}
		\label{fig:Til_walk_1n_first}
	\end{figure}
	
	Following the well-known proof of tiling with squares and dominoes, we distinguish two types of tilings on $(1\times n)$-board. When the last tile is a domino and the number of walks on $(1\times (n-2))$-board is $v_{n-2}$, then we have to finish with two right steps or in case of the all $F_{n-1}$ different tilings, we step up at the end of the board. The left-hand side of Figure~\ref{fig:Til_walk_1n_fibo} shows these cases.  In these subfigures, the small circle illustrates the end of the walks on $(1\times (n-2))$-board.
	Similarly, (the right-hand side of Figure~\ref{fig:Til_walk_1n_fibo}), when the last tile is a square, then the number of possible walks is $v_{n-1}+F_n$. 
	Summarizing all the walks for $n\geq2$, we have 
	\begin{equation}\label{eq:1xn_1}
		v_{n}=v_{n-2}+F_{n-1}+v_{n-1}+F_{n}=v_{n-1}+v_{n-2}+F_{n+1}.
	\end{equation}
	
	\begin{figure}[!ht]
		\centering
		\includegraphics[scale=0.9]{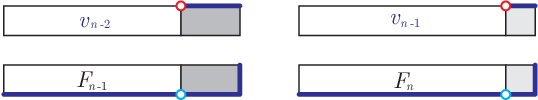} 
		\caption{Tiling and walks on $(1\times n)$-board with recurrence}
		\label{fig:Til_walk_1n_fibo}
	\end{figure}
	
	On the other hand for the proof of Theorem~\ref{th:main_1xn}, using exactly $k$ dominoes, the number of walks of all the $n$-tilings is the same, exactly $n-k+1$ (see Figure~\ref{fig:Til_walk_1n}).	
	
	\begin{figure}[!ht]
		\centering
		\includegraphics[scale=0.9]{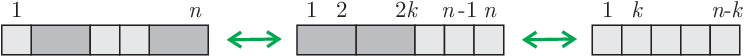} 
		\caption{Tiling with exactly $k$ dominoes on $(1\times n)$-board}
		\label{fig:Til_walk_1n}
	\end{figure}
	
	\begin{lemma}[Benjamin and Quinn \cite{BQ}]
		The number of $n$-tilings using exactly $k$ dominoes is 
		\[\binom{n-k}{k}, \qquad (k=0,1,\ldots,\left\lfloor n/2\right\rfloor),\]
		where $\lfloor n\rfloor$ is the integer part of $n$.
	\end{lemma}	
	Thus, because $k=0,1,\ldots,\lfloor n/2\rfloor$, we obtain that $v_n=\sum_{k=0}^{\lfloor n/2\rfloor}\binom{n-k}{k}(n-k+1)$. Finally, using identity $\sum_{k=0}^{\lfloor n/2\rfloor}\binom{n-k}{k}=F_{n+1}$ \cite[p.~155, Eq.~12.1]{Koshy}, after the following calculations
	\begin{eqnarray*}
		v_n&=&\sum_{k=0}^{\left\lfloor n/2\right\rfloor}\binom{n-k}{k}(n-k+1)=  
		(n+1)\sum_{k=0}^{\left\lfloor n/2\right\rfloor}\binom{n-k}{k}-\sum_{k=0}^{\lfloor n/2\rfloor}k\binom{n-k}{k}\\
		&=&(n+1)F_{n+1}- \sum_{k=1}^{\lfloor n/2\rfloor}k\binom{n-k}{k} =
		(n+1)F_{n+1}- \sum_{k=0}^{\lfloor (n-2)/2\rfloor}(k+1)\binom{n-(k+1)}{k+1}  \\
		&=&(n+1)F_{n+1}- \sum_{k=0}^{\lfloor (n-2)/2\rfloor}(k+1)\frac{((n-2)-k+1)((n-2)-k)!}{(k+1)k!((n-2)-2k)!}\\	
		&=&(n+1)F_{n+1}- \sum_{k=0}^{\lfloor (n-2)/2\rfloor}\binom{n-2-k}{k}(n-2-k+1) \\
		&=& (n+1)F_{n+1}-v_{n-2},
	\end{eqnarray*}
	our new relation is 
	\begin{equation}\label{eq:1xn_2}
		v_n+v_{n-2}=(n+1)F_{n+1}.
	\end{equation}
	
	Eliminating the Fibonacci numbers from \eqref{eq:1xn_1} and \eqref{eq:1xn_2} we have $v_n+v_{n-2}=(n+1)(v_{n}-v_{n-1}-v_{n-2})$, and it proves Theorem~\ref{th:main_1xn}.
	
	\begin{proof}[Proof of Corollary \ref{cor:main_1xn}]
		We find the recurrence relation of the sequence $v_n$ in form 
		\begin{equation}\label{eq:ABCD}
			v_{n}=Av_{n-1}+Bv_{n-2}+Cv_{n-3}+Dv_{n-4}, \quad n\geq4,
		\end{equation}
		where $A, B, C$, and $D$ are constant real numbers. Now, we express $v_n$ and $v_{n-4}$ from recurrence relations \eqref{eq:main_1xn} and its shifted version $ (n-2)v_{n-2}=(n-1)v_{n-3}+nv_{n-4}$ for $n\geq 4$, respectively, and substitute them into \eqref{eq:ABCD}. We then have
		\begin{equation*} 
			\def\arraystretch{1.3}
			\begin{array}{rcl}
				\frac{n+1}{n}v_{n-1}+\frac{n+2}{n}v_{n-2}&=&Av_{n-1}+Bv_{n-2}+Cv_{n-3} +D(\frac{n-2}{n}v_{n-2}-\frac{n-1}{n}v_{n-3}),\\
				\frac{(1-A)n+1}{n}v_{n-1}&=&(-\frac{n+2}{n}+\frac{Bn}{n}+\frac{Dn-2D}{n})v_{n-2}+(\frac{Cn}{n}-\frac{Dn-D}{n})v_{n-3},\\
				((1-A)n+1)v_{n-1}&=&((B+D-1)n-2(D+1))v_{n-2}+((C-D)n+D)v_{n-3},\\
				((A-1)n-1)v_{n-1}&=&((1-B-D)n+2(D+1))v_{n-2}+((D-C)n-D)v_{n-3}.
			\end{array} 
		\end{equation*}
		Since from \eqref{eq:main_1xn}  we also have
		\[(n-1)v_{n-1}=nv_{n-2}+(n+1)v_{n-3},\]
		then from the equalities of the coefficients we obtain the system of linear sequences
		\begin{equation*} 
			\def\arraystretch{1.3}
			\begin{array}{rcl}
				n-1&=&(A-1)n-1,\\
				n&=&(1-B-D)n+2(D+1),\\
				n+1&=&(D-C)n-D,
			\end{array} 
		\end{equation*}	
		where the solutions are $A=2, B=1, C=-2$, and $D=-1$.        
		Thus, for  $n\geq4$ the recurrence \eqref{eq:ABCD} holds. By a similar calculation, we can verify that there is no third-order recursion for $v_n$.
		
		For the proof of the second part of the corollary, first, we sum the equations \eqref{eq:1xn_1} and \eqref{eq:1xn_2}, then subtract the shifted versions (in $n$ by one) of  \eqref{eq:1xn_1} and \eqref{eq:1xn_2}. Thus  we obtain, respectively, 
		\begin{equation*} 
			\def\arraystretch{1.3}
			\begin{array}{rcl}
				2v_{n}-v_{n-1}&=&(n+2)F_{n+1},\\
				v_{n}+2v_{n-1}&=&(n+1)F_{n+2}.
			\end{array} 
		\end{equation*}	
		Finally, the sum of twice of the first equation and the second one gives the statement.		
	\end{proof}

	\section{Tiling and walking on \texorpdfstring{$(2\times n)$} --board}
	
	In this section, we prove the Theorem \ref{th:main_2xn}. Even though the number of $n$-tilings on $(2\times n)$-board is known (i.e., \cite{ML}), first, we will give a proof for it. This is because we will use later the method and the recurrence sequences which are defined during the proof. Second, we give recurrence sequences, which together give the number of walks. Finally, we solve the derived system of the recurrence sequences to obtain the final recurrence form \eqref{th:main_2xn} by using linear algebra.

	Let $r_n$ be the number of tilings with squares and dominoes on $(2\times n)$-board. We  distinguish three different types of not finished tilings on $(2\times n)$-board. According to the first column of Figure~\ref{fig:Til_walk_tiling_2n}, 
	type $A$ is when the last lower tile is a domino and the last upper tile, the square, is missing; 
	type $C$ is when the last lower square is missing;
	type $D$ is when the last upper tile is a domino and the last lower tile, a horizontal domino, which could be two squares, is missing.
	Let $a_n$, $c_n$, and $d_n$ be the numbers of their elements, respectively. If we put tiles into the missing parts, then we can build our tilings recursively. For example, in the third row, we can get tilings of type $C$ putting squares and dominoes from full-, type $A$ or type $D$ tilings. 
	
	\begin{figure}[!ht]
		\centering
		\includegraphics[scale=0.9]{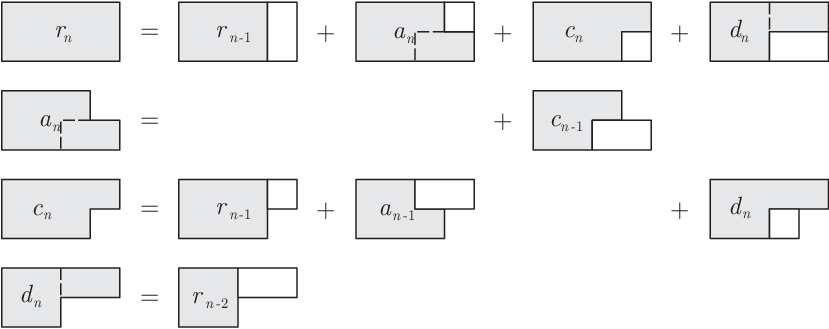} 
		\caption{Tiling with squares and dominoes on $(2\times n)$-board}
		\label{fig:Til_walk_tiling_2n}
	\end{figure}
	Figure~\ref{fig:Til_walk_tiling_2n} shows how we can obtain the tilings recursively, and it implies the system of recurrence equations \eqref{eq:sys_rabcd} for $n\geq2$,      
	\begin{equation}\label{eq:sys_rabcd}
		\begin{array}{rcl}
			r_n&=&r_{n-1}+a_n+c_n+d_n,\\
			a_n&=&c_{n-1},\\
			c_n&=&r_{n-1}+a_{n-1}+d_n,\\
			d_n&=&r_{n-2},
		\end{array} 
	\end{equation}
	where the initial values are $a_0=c_0=d_0=a_1=d_1=0$, $r_0=1$, $r_1=2$, and $c_1=1$ (see Figure~\ref{fig:Til_walk_tiling_init}).
	
	\begin{figure}[!ht]
		\centering
		\includegraphics[scale=0.9]{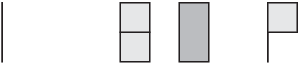} 
		\caption{Initial terms of tiling with squares and dominoes}
		\label{fig:Til_walk_tiling_init}
	\end{figure}
	
	To solve the system \eqref{eq:sys_rabcd}, we eliminate $a_n$ and $d_n$. Then we have  
	\begin{equation}\label{eq:sys_rc}
		\begin{array}{rcl}
			r_n&=&r_{n-1}+r_{n-2}+c_n+c_{n-1},\\
			c_n&=&r_{n-1}+r_{n-2}+c_{n-2}.
		\end{array}
	\end{equation}
	The sum of the equations from \eqref{eq:sys_rc} is
	\[ 	r_{n}=2r_{n-1}+2r_{n-2}+c_{n-1}+c_{n-2}. \]
	Combining its shifted version 
	\[ 	r_{n+1}=2r_{n}+2r_{n-1}+c_{n}+c_{n-1} \]
	with the first equation of \eqref{eq:sys_rc}  we have
	\[	r_{n+1}-r_n = 2r_{n}+2r_{n-1}+c_{n}+c_{n-1} -(r_{n-1}+r_{n-2}+c_n+c_{n-1})=2r_n+r_{n-1}-r_{n-2}. \]
	Thus, our recurrence relation for the sequence $(r_n)$ is 
	\begin{equation}\label{eq:r_recu}
		r_{n}= 3r_{n-1}+r_{n-2}-r_{n-3} \quad (n\geq 3),
	\end{equation}
	with initial values $r_0=1$, $r_1=2$, and $r_2=7$. It appears in OEIS \cite{oeis} as sequence {A030186}.
	
	By the help of the system of equations \eqref{eq:sys_rabcd} we fill Table~\ref{tab:til01}, and we find a simple relationship between sequences $(r_n)$ and $(c_n)$ as given in the following lemma. 
	\begin{lemma}\label{lem:rcc}
		For $n\geq0$, 
		\begin{equation}\label{eq:rcc}
			r_n=c_{n+1}-c_n.
		\end{equation}
	\end{lemma}
	\begin{proof} The proof is by induction on $n$. For $n=0$ and $n=1$ the statement holds, see Table~\ref{tab:til01}. Let us  suppose that the result is true for all $i\leq n-1$. Then in particular, the result for $i=n-2$ says that  $c_{n-1}-c_{n-2}=r_{n-2}$, and from the second equation of system \eqref{eq:sys_rc} we have 
		$c_{n+1}-c_{n}=r_{n}+r_{n-1}+c_{n-1}- (r_{n-1}+r_{n-2}+c_{n-2}) = r_{n} + (c_{n-1}-c_{n-2}-r_{n-2}) =r_{n}.$
	\end{proof}
	\begin{table}[!ht]		\centering
		\ra{1.3}
		\begin{tabular}{@{}cccccccc@{}}\toprule
			$n$     & 0 & 1 & 2 & 3 & 4 & 5 & 6\\ \midrule
			$r_n$ 	& 1 & 2 & 7 & 22 & 71 & 228 & 733 \\
			$a_n$ 	& 0 & 0 & 1 & 3 & 10 & 32 & 103   \\
			$c_n$ 	& 0 & 1 & 3 & 10 & 32 & 103 & 331 \\
			$d_n$ 	& 0 & 0 & 1 & 2 & 7 & 22 & 71  \\
			\bottomrule
		\end{tabular}
		\caption{First few items of sequences $(r_n)$, $(a_n)$, $(c_n)$, and $(d_n)$} 
		\label{tab:til01}
	\end{table}

	We prove that the same recurrence relation holds for $(c_n)$, as for $(r_n)$. 
	\begin{lemma}\label{lem:c_recu}
		For $n\geq 3$, 
		\begin{equation}\label{eq:c_recu}
			c_{n}= 3c_{n-1}+c_{n-2}-c_{n-3}.
		\end{equation}
	\end{lemma}
	\begin{proof}  Taking the difference between the equations of \eqref{eq:sys_rc} we have $r_n=2c_n+c_{n-1}-c_{n-2}$, and with \eqref{eq:rcc} we obtain 
		$ 0 = r_n-r_n= (c_{n+1}- c_{n})- (2c_{n}+c_{n-1}-c_{n-2})=c_{n+1}- 3c_{n}-c_{n-1}+c_{n-2}. $
	\end{proof}

	From now, we are able to determine the number of walks. We use the previously introduced types of tilings according to Figure~\ref{fig:Til_walk_tiling_2n}. They are the $r_n$ finished tilings on $(2\times n)$-board and the types $A$, $C$, and $D$ with number $a_n$, $c_n$, and $d_n$, respectively. We consider the walks on these tilings. The $(2\times n)$-board has three horizontal grid lines, the zero-, first-, and second line. When a walk ends at on these lines we denote it, respectively, by uppercases $^{(0)}$, $^{(1)}$, and $^{(2)}$. Moreover, in the figures we denoted the end points by small cyan, red, and blue circles. 
	Let $r_n^{(0)}$, $r_n^{(1)}$, $r_n^{(2)}$ be all the numbers of walks on  $(2\times n)$-board ending on grid lines 0, 1 and 2, respectively. For example, Figure~\ref{fig:Til_walk_ex_2x2} presents all the walks on the $(2\times 2)$-board in case of a given tiling, where the walks end on different horizontal grid lines.  Furthermore,  one can easily calculate as an exercise that $r_2^{(0)}=r_2=7$, $r_2^{(1)}=14$, $r_2^{(2)}=28$ with drawing all the tilings and their walks on $(2\times 2)$-board.  The equation $r_n^{(0)}=r_n$ holds for all $n\geq0$, and  $r_n^{(2)}$ gives the numbers of our final examined walks, thus, $w_n=r_n^{(2)}$.
	
	\begin{figure}[!ht]
		\centering
		\includegraphics[scale=0.9]{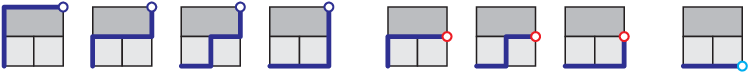} 
		\caption{Walks on $(2\times 2)$-board in case of a given tiling}
		\label{fig:Til_walk_ex_2x2}
	\end{figure}
	
	Additionally, let $a_n^{(0)}$, $a_n^{(1)}$, $a_n^{(2)}$, $c_n^{(0)}$, $c_n^{(1)}$, $c_n^{(2)}$, $d_n^{(0)}$, $d_n^{(1)}$, and $d_n^{(2)}$ be all the numbers of walks ending on grid lines 0, 1, and 2 on unfinished board types $A$, $C$, and $D$, respectively. 	Immediately, we observe that not only $r_n^{(0)}=r_n$  but also $a_n^{(0)}=a_n$, $c_n^{(0)}=c_n$, and $d_n^{(0)}=d_n$,    hold for all $n\geq0$. 
	
	Figure~\ref{fig:Til_walk_2n_init} shows all the possible finished and unfinished tilings and its walks include the cases with superscripts $^{(1)}$ and $^{(2)}$ for $n=0$ and $n=1$, 
	and Figure~\ref{fig:Til_walk_tiling_walks}, based on Figure~\ref{fig:Til_walk_tiling_2n}, displays recursively the evolution of the numbers of walks in case of the different types of boards and grid lines.
	
	\begin{figure}[!ht]
		\centering
		\includegraphics[scale=0.9]{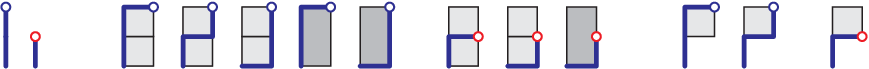} 
		\caption{Initial walks for Figure \ref{fig:Til_walk_tiling_walks}}
		\label{fig:Til_walk_2n_init}
	\end{figure}		
	
	\begin{figure}
		\centering
		\includegraphics[scale=0.9]{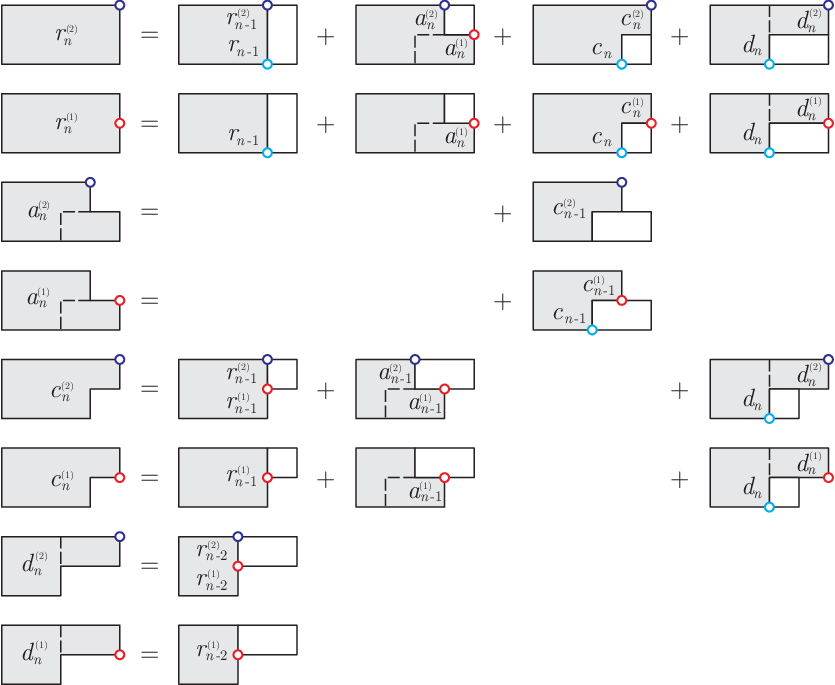} 
		\caption{Recursions of numbers of walks on tiled boards}
		\label{fig:Til_walk_tiling_walks}
	\end{figure}

	Finally, Figure~\ref{fig:Til_walk_tiling_walks} implies the following system of twelve recurrence equations \eqref{eq:sys_r-d} together with recurrences \eqref{eq:sys_rabcd} for $n\geq2$.
	
	\begin{equation}\label{eq:sys_r-d} 
		\def\arraystretch{1.3}
		\begin{array}{rcl}
			r_n^{(2)}&=&(r_{n-1}^{(2)}+r_{n-1})+(a_{n}^{(2)}+a_{n}^{(1)}) +(c_{n}^{(2)}+c_{n})+(d_{n}^{(2)}+d_{n}),\\
			r_n^{(1)}&=&r_{n-1}+a_{n}^{(1)}+(c_{n}^{(1)}+c_{n})+(d_{n}^{(1)}+d_{n}),\\
			a_n^{(2)}&=&c_{n-1}^{(2)},\\
			a_n^{(1)}&=&c_{n-1}^{(1)}+c_{n-1},\\
			c_n^{(2)}&=&(r_{n-1}^{(2)}+r_{n-1}^{(1)})+(a_{n-1}^{(2)}+a_{n-1}^{(1)}) +(d_{n}^{(2)}+d_{n}), \\	
			c_n^{(1)}&=& r_{n-1}^{(1)}+a_{n-1}^{(1)}+(d_{n}^{(1)}+d_{n}),\\
			d_n^{(2)}&=&r_{n-2}^{(2)}+r_{n-2}^{(1)},\\
			d_n^{(1)}&=&r_{n-2}^{(1)}
		\end{array} 
	\end{equation}
	with initial values in Table~\ref{tab:til01}, Table~\ref{tab:til_nnitial}, and Figure~\ref{fig:Til_walk_2n_init}.
	\begin{table}[!ht]\centering
		\ra{1.3}
		\begin{tabular}{@{}cccccccccc@{}}\toprule
			$n$  &   & $r_n^{(2)}$ & $r_n^{(1)}$ & $a_n^{(2)}$  & $a_n^{(1)}$ & $c_n^{(2)}$  & 	$c_n^{(1)}$ & $d_n^{(2)}$ &$d_n^{(1)}$\\ \midrule
			0 & & 1  & 1  & 0 & 0 & 0  & 0 & 0 & 0 \\
			1 & & 5  & 3  & 0 & 0 & 2  & 1 & 0 & 0 \\
			2 & & 28 & 14 & 2 & 2 & 11 & 5 & 2 & 1 \\
			\bottomrule
		\end{tabular}
		\caption{First few terms of sequences  \eqref{eq:sys_r-d} } 
		\label{tab:til_nnitial}
	\end{table}
	
	Recall since $w_n=r_n^{(2)}$, then the final purpose is to give recursively the sequences $r_n^{(2)}$ (if it is possible). Hence, we solve the system \eqref{eq:sys_r-d}.  Firstly, we eliminate sequences $(a_n^{(j)})$ and $(d_n^{(j)})$, $j\in\{1,2\}$. The system becomes
	\begin{eqnarray*}
		r_n^{(2)}&=&r_{n-1}^{(2)}+r_{n-2}^{(2)}+r_{n-2}^{(1)} 
		+c_{n}^{(2)}+c_{n-1}^{(2)}+c_{n-1}^{(1)} +r_{n-1}+r_{n-2} +c_{n}+c_{n-1},\\
		r_n^{(1)}&=& r_{n-2}^{(1)} + c_{n}^{(1)}+ c_{n-1}^{(1)} + r_{n-1}+r_{n-2} + c_{n}+ c_{n-1},\\
		c_n^{(2)}&=&r_{n-1}^{(2)}+r_{n-2}^{(2)} +r_{n-1}^{(1)}+r_{n-2}^{(1)}   +c_{n-2}^{(2)}+c_{n-2}^{(1)} +r_{n-2} +c_{n-2},\\	
		c_n^{(1)}&=& r_{n-1}^{(1)}+r_{n-2}^{(1)}+c_{n-2}^{(1)} +r_{n-2}+c_{n-2}.
	\end{eqnarray*}
	After using  the first equation of \eqref{eq:sys_rc}, we have  
	\begin{equation}\label{eq:sys_4eq}\def\arraystretch{1.3}
		\begin{array}{rcl}
			r_n^{(2)}&=&r_{n-1}^{(2)}+r_{n-2}^{(2)}+r_{n-2}^{(1)} 
			+c_{n}^{(2)}+c_{n-1}^{(2)}+c_{n-1}^{(1)} +r_{n},\\
			r_n^{(1)}&=& r_{n-2}^{(1)} + c_{n}^{(1)}+ c_{n-1}^{(1)} + r_{n},\\
			c_n^{(2)}&=&r_{n-1}^{(2)}+r_{n-2}^{(2)} +r_{n-1}^{(1)}+r_{n-2}^{(1)}   +c_{n-2}^{(2)}+c_{n-2}^{(1)} +r_{n-2} +c_{n-2},\\	
			c_n^{(1)}&=& r_{n-1}^{(1)}+r_{n-2}^{(1)}+c_{n-2}^{(1)} +r_{n-2}+c_{n-2}.
		\end{array}
	\end{equation}	
	The differences of the first two and the last two equations with simplification are  
	\begin{equation}\label{eq:sys_a0}\def\arraystretch{1.3}
		\begin{array}{rcl}
			r_n^{(2)}-r_n^{(1)}&=&r_{n-1}^{(2)}+r_{n-2}^{(2)} 
			+c_{n}^{(2)}+c_{n-1}^{(2)}- c_{n}^{(1)},\\
			c_n^{(2)}- c_n^{(1)}&=&r_{n-1}^{(2)}+r_{n-2}^{(2)}  +c_{n-2}^{(2)}.	
		\end{array}
	\end{equation}
	From the second equation of \eqref{eq:sys_a0}, we substitute $c_n^{(1)}$ into the first one, and after rearranging we obtain 
	\begin{equation}\label{eq:sys_a1}\def\arraystretch{1.3}
		\begin{array}{rcl}
			r_n^{(1)}&=&r_n^{(2)}-2r_{n-1}^{(2)}-2r_{n-2}^{(2)} - c_{n-1}^{(2)}- c_{n-2}^{(2)},  \\
			c_n^{(1)}&=& c_n^{(2)} -r_{n-1}^{(2)} -r_{n-2}^{(2)}   -c_{n-2}^{(2)}.
		\end{array}
	\end{equation}
	Now, we manage to eliminate the sequences $(r_n^{(1)})$ and $(c_n^{(1)})$ from \eqref{eq:sys_4eq}. We shift the equations of \eqref{eq:sys_a1} by $-1$ and $-2$ in the subscript. For example, $c_{n-1}^{(1)}= c_{n-1}^{(2)} -r_{n-2}^{(2)} -r_{n-3}^{(2)}   -c_{n-3}^{(2)}$, we substitute them into the first and third rows of \eqref{eq:sys_4eq}. Hence,
	\begin{equation}\label{eq:sys_a2}\def\arraystretch{1.3}
		\begin{array}{rclll}
			r_n^{(2)}&=&r_{n-1}^{(2)}+r_{n-2}^{(2)}-3r_{n-3}^{(2)}-2r_{n-4}^{(2)} 
			& +c_{n}^{(2)}	+2c_{n-1}^{(2)}  - 2c_{n-3}^{(2)} -c_{n-4}^{(2)} & +r_{n},\\
			c_n^{(2)}&=&2r_{n-1}^{(2)}-5r_{n-3}^{(2)}-3r_{n-4}^{(2)} 
			&+c_{n-2}^{(2)}		- 2c_{n-3}^{(2)}  -2c_{n-4}^{(2)} &+r_{n-2} +c_{n-2}.
		\end{array}
	\end{equation}
	Recall the equations of \eqref{eq:sys_rc} hold. Using Lemmas \ref{lem:rcc} and \ref{lem:c_recu}, our reduced system of recurrence equations is  
	\begin{equation}\label{eq:sys_3eq}\def\arraystretch{1.3}
		\begin{array}{rclll}
			r_n^{(2)}&=&r_{n-1}^{(2)}+r_{n-2}^{(2)}-3r_{n-3}^{(2)}-2r_{n-4}^{(2)} 
			& +c_{n}^{(2)}	+2c_{n-1}^{(2)}  - 2c_{n-3}^{(2)} -c_{n-4}^{(2)} & +c_{n+1}-c_n,\\
			c_n^{(2)}&=&2r_{n-1}^{(2)}-5r_{n-3}^{(2)}-3r_{n-4}^{(2)} 
			&+c_{n-2}^{(2)}		- 2c_{n-3}^{(2)}  -2c_{n-4}^{(2)} &+c_{n-1},\\ 
			c_n&=& &3c_{n-1}+c_{n-2}-c_{n-3}.
		\end{array}
	\end{equation}

	In the following part, we eliminate the sequence $(c_n)$. Shift the second equation of \eqref{eq:sys_3eq} by $1$ and $2$ in  the subscript, and multiply by $-1$ the second shifted one, producing equations of the form. Then $c_{n+1}^{(2)}=\cdots +c_n$ and $-c_{n+2}^{(2)}=\cdots -c_{n+1}$. Summing them with the first equation of \eqref{eq:sys_3eq}, and after shortening we gain   
	\begin{equation*}
		2r_{n+1}^{(2)}-r_n^{(2)}-6r_{n-1}^{(2)}	+ r_{n-2}^{(2)} +6r_{n-3}^{(2)} +2r_{n-4}^{(2)}= 
		c_{n+2}^{(2)} -c_{n+1}^{(2)} +5c_{n-1}^{(2)}	-4c_{n-3}^{(2)}-c_{n-4}^{(2)}.
	\end{equation*}
	Let its shifted version be
	\begin{equation}\label{eq:sys_As}
		L_{n}^{A} = R_{n+1}^{A},
	\end{equation}
	where
	\begin{eqnarray*}\label{eq:sys_A}
		L_{n}^{A} &=& 2r_{n}^{(2)}-r_{n-1}^{(2)}-6r_{n-2}^{(2)}	+ r_{n-3}^{(2)} +6r_{n-4}^{(2)} +2r_{n-5}^{(2)}\\
		R_{n+1}^{A} &=&  c_{n+1}^{(2)} -c_{n}^{(2)} +5c_{n-2}^{(2)}	-4c_{n-4}^{(2)}-c_{n-5}^{(2)}.
	\end{eqnarray*}
	Let us shift the second equations of \eqref{eq:sys_3eq} again by $-1, -2$, and $1$ in  the subscript, and express $c_n$, $c_{n-1}$, $c_{n-2}$, and $c_{n-3}$ from them and the original one. Then substitute them into the third equation of \eqref{eq:sys_3eq}, and after shortening we gain   
	\begin{multline}\label{eq:sys_B}
		2r_{n}^{(2)}-6r_{n-1}^{(2)}-7r_{n-2}^{(2)} +14r_{n-3}^{(2)} +14r_{n-4}^{(2)} -2r_{n-5}^{(2)} -3r_{n-6}^{(2)} =\\
		c_{n+1}^{(2)}- 	3 c_n^{(2)} -2c_{n-1}^{(2)}	+ 6c_{n-2}^{(2)} -3c_{n-3}^{(2)} -9c_{n-4}^{(2)} +2c_{n-6}^{(2)}.
	\end{multline}
	Let  $B$ denote the equation \eqref{eq:sys_B}, in short form  
	\begin{equation*}\label{eq:sys_Bs}
		L_{n}^{B} = R_{n+1}^{B},
	\end{equation*}
	where the denotations are similar to \eqref{eq:sys_As}.

	In the succeeding part, by using the linear algebra we solve the system of 	
	\begin{equation}\label{eq:sys_RARB}\def\arraystretch{1.3}
		\begin{array}{rcl}
			L_{n}^{A}&=& 	R_{n+1}^{A},\\
			L_{n}^{B}&=& 	R_{n+1}^{B}.
		\end{array}
	\end{equation}
	
	Let us consider the 12-dimensional vector space, where a basis is defined by  $c_{n+1}^{(2)}$, $c_{n}^{(2)}$, $c_{n-1}^{(2)}$, \ldots, $c_{n-9}^{(2)}$, and $c_{n-10}^{(2)}$. Then the coordinates of $R_{n+1}^{A}$ and its shifted versions by $-1$, $-2$, $\ldots$, $-5$ are 
	\begin{equation*}\def\arraystretch{1.3}
		\begin{array}{rcl}
			R_{n+1}^{A}&=& 	(1, -1, 0, 5, 0, -4, -1, 0, 0, 0, 0,0),\\
			R_{n}^{A}&=& 	(0,1, -1, 0, 5, 0, -4, -1, 0, 0, 0, 0),\\
			& \vdots & \\
			R_{n-4}^{A}&=& 	(0, 0, 0, 0, 0, 1, -1, 0, 5, 0, -4, -1),
		\end{array}
	\end{equation*}
	respectively. 
	Similarly,

	\begin{equation*}\def\arraystretch{1.3}
		\begin{array}{rcl}
			R_{n+1}^{B}&=& 	(1, -3, -2, 6, -3, -9, 0, 2, 0, 0, 0, 0),\\
			R_{n}^{B}&=& 	(0,1, -3, -2, 6, -3, -9, 0, 2, 0, 0, 0),\\
			& \vdots & \\
			R_{n-3}^{B}&=& 	(0, 0, 0, 0,1, -3, -2, 6, -3, -9, 0, 2).
		\end{array}
	\end{equation*}
	
	We define the linear combinations $\alpha_0 R_{n+1}^{A} + \alpha_1 R_{n}^{A}+\cdots +\alpha_5 R_{n-4}^{A}$ and $\beta_0 R_{n+1}^{B} + \beta_1 R_{n}^{B}+\cdots +\beta_4 R_{n-3}^{B}$. Now, our question is: whether they  could be equal and for what $\alpha_j$ and $\beta_k$? For the answer, we have to solve the homogeneous linear vector equation   
	\[\sum_{j=0}^{5}\alpha_jR_{n-j+1}^{A} - \sum_{j=0}^{4}\beta_jR_{n-j+1}^{B}=\mathbf{0},\]
	where $\mathbf{0}$ is the zero vector. 
	Its matrix form is 
	\[ \mathbf{M} \mathbf{x}=\mathbf{0},\]
	where $\mathbf{x}=(\alpha_0,\ldots, \alpha_5,\beta_0,\ldots, \beta_4)$ and
	\[\def\arraystretch{0.1}\arraycolsep=2pt
	\mathbf{M}=
	\left[ \begin {array}{cccccc|ccccc} 1&0&0&0&0&0&-1&0&0&0&0
	\\ \noalign{\medskip}-1&1&0&0&0&0&3&-1&0&0&0\\ \noalign{\medskip}0&-1&
	1&0&0&0&2&3&-1&0&0\\ \noalign{\medskip}5&0&-1&1&0&0&-6&2&3&-1&0
	\\ \noalign{\medskip}0&5&0&-1&1&0&3&-6&2&3&-1\\ \noalign{\medskip}-4&0
	&5&0&-1&1&9&3&-6&2&3\\ \noalign{\medskip}-1&-4&0&5&0&-1&0&9&3&-6&2
	\\ \noalign{\medskip}0&-1&-4&0&5&0&-2&0&9&3&-6\\ \noalign{\medskip}0&0
	&-1&-4&0&5&0&-2&0&9&3\\ \noalign{\medskip}0&0&0&-1&-4&0&0&0&-2&0&9
	\\ \noalign{\medskip}0&0&0&0&-1&-4&0&0&0&-2&0\\ \noalign{\medskip}0&0&0
	&0&0&-1&0&0&0&0&-2\end {array} \right].
	\]
	After solving it we find that the solution vector with a free parameter $t$ is 
	\[ \mathbf{x}=(t,-5t,7t,-3t,-4t,2t, t,-3t,5t,-2t,-t).\]
	(We mention, since we got only one free parameter, then we must not reduce the dimension of our vector space.) Put $t=1$, then
	\[
	R_{n+1}^{A} -5R_{n}^{A}+ 7R_{n-1}^{A} -3R_{n-2}^{A} -4R_{n-3}^{A} +2R_{n-4}^{A} =	R_{n+1}^{B} -3R_{n}^{B} +5R_{n-1}^{B} -2R_{n-2}^{B} -R_{n-3}^{B}.
	\]
	It means from \eqref{eq:sys_RARB} that the expressions $\sum_{j=0}^{5}\alpha_jL_{n-j+1}^{A}$ and $\sum_{j=0}^{4}\beta_jL_{n-j+1}^{B}$ are also equal. Thus, after some calculations, we obtain 
	\begin{equation*}
		r_{n-1}^{(2)}-8r_{n-2}^{(2)}+ 17 r_{n-3}^{(2)} +7r_{n-4}^{(2)} -41r_{n-5}^{(2)} -r_{n-6}^{(2)} +23r_{n-7}^{(2)} -3r_{n-8}^{(2)} -4r_{n-9}^{(2)} +r_{n-10}^{(2)}	= 0,
	\end{equation*}
	which proves the first part of our main Theorem~\ref{th:main_2xn}. Recall $w_n=r_n^{(2)}$.
	
	Now, we show a connection between the recurrences of the tiling sequence $r_n$ and tiling-walking sequence $w_n$ given by the equations \eqref{eq:rcc} and  \eqref{eq:main_2xn}, respectively. Their characteristic polynomials are 
	\begin{equation*}
		p_{r_n}(x)=x^3-3x^2-x+1
	\end{equation*}
	and
	\begin{multline*}
		p_{w_n}(x)=x^9-8x^8+17x^7+7x^6-41x^5-x^4+23x^3-3x^2-4x+1\\ =(x+1)(x^2-3x+1)(x^3-3x^2-x+1)^2,
	\end{multline*}
	respectively. It shows that $w_n$ is en extension of $r_n$, moreover, considering the factorization of $p_{w_n}$ let us define the following functions $y_n$, $x_n$, and $\widetilde{r}_{n}$ recursively
	\begin{equation*}
		\begin{array}{rcl}
			y_{n}&=& w_n+w_{n-1},\\
			x_{n}&=& y_n-3y_{n-1}+y_{n-2},\\
			\widetilde{r}_{n}&=& x_n-3x_{n-1}-x_{n-2}+x_{n-3}
		\end{array}
	\end{equation*}
	using the recurrences  $w_n=-w_{n-1}$, $y_n=3y_{n-1}-y_{n-2}$, and  $x_n=3x_{n-1}+x_{n-2}-x_{n-3}$ with characteristic polynomials $x+1$, $x^2-3x+1$, and $x^3-3x^2-x+1$, respectively.
	Substitute $y_n$ into function $x_n$, then do $x_n$ into $\widetilde{r}_{n}$, and finally, substitute $\widetilde{r}_{n}$ into the recurrence equation \eqref{eq:rcc} instead of $r_n$, then with an easy calculation the result is \eqref{eq:main_2xn}.
	It proves the second  part of Theorem~\ref{th:main_2xn}.

	\subsection{Tiling and walking with only dominoes}
	
	In this subsection, we allow only dominoes for the tilings, and we prove Corollary~\ref{th:maincor_2xn}. In this case, we have to omit all the sequences $(a_n^{(j)})$ and $(c_n^{(j)})$, $j\in\{0, 1,2\}$, from systems \eqref{eq:sys_rabcd} and \eqref{eq:sys_r-d} since they contain tilings with squares. Thus, after some trivial substitution for $n\geq2$, we have 
	\begin{equation}\label{eq:sys_onlydom01} 
		\def\arraystretch{1.3}
		\begin{array}{rcl}
			r_n^{(2)}&=&r_{n-1}^{(2)}+r_{n-2}^{(2)}+r_{n-2}^{(1)}+r_{n},\\
			r_n^{(1)}&=&r_{n-2}^{(1)}+r_{n},\\
			r_n&=&r_{n-1}+r_{n-2}
		\end{array} 
	\end{equation}
	with initial values $r_0=1$, $r_1=1$, $r_0^{(2)}=1$, $r_1^{(2)}=2$, $r_0^{(1)}=1$, and $r_1^{(1)}=1$ (see Figure~\ref{fig:Til_walk_2n_init_domino}).
	From the first equation we express $r_{n}$, and we substitute it and its shifted versions into the third equation, moreover, subtract the second equation from the first, and we have the equations	
	\begin{equation}\label{eq:sys_onlydom02} 
		\def\arraystretch{1.3}
		\begin{array}{rcl}
			r_n^{(2)}-2r_{n-1}^{(2)}-r_{n-2}^{(2)}+2r_{n-3}^{(2)}+r_{n-4}^{(2)}&=&  r_{n-2}^{(1)}-r_{n-3}^{(1)}-r_{n-4}^{(1)},\\	r_n^{(2)}-r_{n}^{(1)}&=&r_{n-1}^{(2)}+r_{n-2}^{(2)}.
		\end{array} 
	\end{equation}
	Finally, from the second equation of \eqref{eq:sys_onlydom02} we express $r_{n}^{(1)}$ and substitute it and its shifted versions into the first one, then we obtain the final recurrence relation 	
	\begin{equation}\label{eq:seq_dom}
		r_n^{(2)}= 2r_{n-1}^{(2)}+2r_{n-2}^{(2)}-4r_{n-3}^{(2)}-2r_{n-4}^{(2)}+2r_{n-5}^{(2)}+r_{n-6}^{(2)}
	\end{equation}
	The initial values can be determined from \eqref{eq:sys_onlydom01}.
	
	\begin{figure}[!ht]
		\centering
		\includegraphics[scale=0.9]{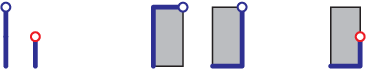} 
		\caption{Initial walks on domino tilings}
		\label{fig:Til_walk_2n_init_domino}
	\end{figure}
	
	The characteristic polynomial of \eqref{eq:seq_dom} is $(x-1)(x+1)(x^2-x-1)^2$ with roots $\alpha=(1+\sqrt{5})/2$, $\beta=(1-\sqrt{5})/2$,  and $\pm 1$, where the multiplicity of $\alpha$ and $\beta$ is 2. Because of the theorem of recursive sequences {(\cite[p.33]{sho})},  for all $n\geq 0$, the explicit form of $r_n^{(2)}$ is given by
	\begin{equation}\label{eq:dom_rec}
		r_n^{(2)}= A\cdot 1^n +B\cdot (-1)^n + (C+D\cdot n)\alpha^n + (E+F\cdot n)\beta^n,
	\end{equation}
	where $A, B, C, D, E, F$ are real constants.
	For $0\leq n\leq5$, the equations \eqref{eq:dom_rec} form a system of linear equations and it solutions are $A=1/2$, $B=1/2$, $C=3\sqrt5/25$, $D=(2+\sqrt 5)/5$, $E=-3\sqrt5/25$, and $F=(2-\sqrt 5)/5$. Thus, the exact explicit form is   
	\begin{equation}\label{eq:expl_rn}
		r_n^{(2)}= \frac12 +\frac{(-1)^n}{2} + \left(\frac{3\sqrt5}{25}  + \frac{2+\sqrt5}{5} n\right)\alpha^n +  \left(\frac{-3\sqrt5}{25}  + \frac{2-\sqrt5}{5} n\right)\beta^n.
	\end{equation}
	Since the $n$th term of the Fibonacci sequence is expressed by $\alpha$ and $\beta$ as follows
	\[F_n= \frac{\alpha^n-\beta^n}{\sqrt5},\]
	where $\alpha+\beta=1$, then we can express $r_n^{(2)}$ by the terms of the Fibonacci sequence. First, we show that  
	\begin{eqnarray*}
		2F_{n+1}-F_n &=& 2 \frac{\alpha\cdot{\alpha}^n - \beta\cdot\beta^n}{\sqrt5} -\frac{{\alpha} ^n - \beta ^n}{\sqrt5} = \frac{1}{\sqrt5}\left((2\alpha-1)\alpha^n+(1-2\beta)\beta^n )\right)\\
		&=&  \frac{1}{\sqrt5}\left((2\alpha-1)\alpha^n+(1-2(1-\alpha))\beta^n )\right)\\
		&=& \frac{2\alpha-1}{\sqrt5}(\alpha^n+\beta^n)=\frac{1+\sqrt5-1}{\sqrt5}(\alpha^n+\beta^n) =\alpha^n+\beta^n.
	\end{eqnarray*} 
	Then from \eqref{eq:expl_rn} we have
	\begin{eqnarray*}
		r_n^{(2)}&=& \frac{1+(-1)^n}{2}  + \frac{3\sqrt5}{25} (\alpha^n -\beta^n ) + 
		\left(\frac{2+\sqrt5}{5} \alpha^n + \frac{2-\sqrt5}{5} \beta^n \right) n\\
		&=& \frac{1+(-1)^n}{2}  + \frac{3}{5}\cdot  \frac{\alpha^n -\beta^n}{\sqrt5} + 
		\left( \frac{\alpha^n - \beta^n}{\sqrt5} +\frac{2}{5}( \alpha^n +  \beta^n)  \right) n\\
		&=&  \frac{1+(-1)^n}{2}  + \left(\frac{3}{5}+n \right)F_n+\frac{2}{5}( 2F_{n+1}-F_n)  n.
	\end{eqnarray*}
	Thus
	\begin{equation*}\label{eq:rn_Fibo_1}
		r_n^{(2)}= \frac{1+(-1)^n}{2}  + \frac{3}{5}(1+n)F_n + \frac{4n}{5} F_{n+1}.
	\end{equation*}
	Finally, for the even and odd indexes, we obtain
	\begin{eqnarray*}\label{eq:rn_Fibo_2}
		r_{2n}^{(2)} &=&   \frac{3}{5}(1+2n)F_{2n} + \frac{8n}{5} F_{2n+1}+1,\\
		r_{2n+1}^{(2)} &=&  \frac{3}{5}(2+2n)F_{2n+1} + \frac{4+8n}{5} F_{2n+2}.
	\end{eqnarray*} 
	This way, as $r_{n}=w_n$, we proved Corollary \ref{th:maincor_2xn}.
	
	Furthermore, from the explicit form \eqref{eq:expl_rn} we obtain that $w_n$ is asymptotic to the dominant part, so 
	\begin{equation}\label{eq:expl_asym}
		w_n\sim  \frac{1+(-1)^n}{2} + \left(\frac{3\sqrt5}{25}  + \frac{2+\sqrt5}{5} n\right)\alpha^n, \qquad \text{as }n\rightarrow \infty.
	\end{equation}
	The convergence is so quick, that for all $n\geq0$, the equation
	\begin{equation*}
		w_n  =\left\lceil \left(\frac{3\sqrt5}{25}  + \frac{2+\sqrt5}{5} n\right)\alpha^n \right\rceil
	\end{equation*}
	holds, where $\lceil . \rceil$ is the ceil function.
	
	
	

	\bibliography{Tiling_walking_bib.bib} 
	\bibliographystyle{plainurl}

\end{document}